\documentclass[12pt]{amsart}
\usepackage{amsfonts}
\usepackage{amsmath}

\newcommand{\Hp}{\ensuremath{ H(P)  }}
\newcommand{\Ha}{\ensuremath{ H(\alpha)  }}
\newcommand{\ZT}{\ensuremath{\mathbb{Z}[T]}}
\newcommand{\R}{\ensuremath{\mathbb{R}}}
\newcommand{\alphabar}{\ensuremath{ \overline{\alpha}  }}

\newcommand{\RTn}{\ensuremath{ \R[T]_{\leq n}  }}
\newcommand{\ZTn}{\ensuremath{ \ZT_{\leq n}  }}
\newcommand{\Cq}{\ensuremath{    \mathcal{C}(q)  }}
\newcommand{\cP}{\ensuremath{    \mathcal{P}  }}

\newtheorem{theorem}{Theorem}[section]
\newtheorem*{theoremA}{Theorem A}
\newtheorem*{theoremB}{Theorem B}
\newtheorem{proposition}[theorem]{Proposition}

\theoremstyle{remark}
\newtheorem{remark}{Remark}

\theoremstyle{definition}

\begin{document}

\baselineskip=17pt

\title[Simultaneous approximation by conjugates of an
algebraic number] {Simultaneous approximation \\ of a real number by
all conjugates \\ of an algebraic number}

\author{Guillaume ALAIN}
\address{
   D\'epartement de Math\'ematiques\\
   Universit\'e d'Ottawa\\
   585 King Edward\\
   Ottawa, Ontario K1N 6N5, Canada}
\email{gyomalin@gmail.com} 

\begin{abstract}{Using a method of H. Davenport and W. M. Schmidt, we
show that, for each positive integer n, the ratio $2/n$ is the
optimal exponent of simultaneous approximation to real irrational
numbers 1) by all conjugates of algebraic numbers of degree $n$, and
2) by all but one conjugates of algebraic integers of degree $n+1$.}
\end{abstract}

\thanks{Work partially supported by NSERC}

\maketitle

\vspace{-1cm}
\section{Introduction}

An outstanding problem in Diophantine approximation, motivated
initially by Mahler's and Koksma's classifications of numbers, is to
provide sharp estimates for the approximation of a real number by
algebraic numbers of bounded degree. Starting with the pioneer work
\cite{Wi} of E.~Wirsing in 1961, this problem has been studied by
many authors and extended in several directions. A good account of
this can be found in Chapter 3 of \cite{Bu}. For our purpose, let us
simply mention that, in 1969, H.~Davenport and W.~M.~Schmidt gave
estimates for the approximation by algebraic integers \cite{DS} and
that, more recently, D.~Roy and M.~Waldschmidt looked at
simultaneous approximations by several conjugate algebraic integers
\cite{RW}. While the latter work was limited to at most one quarter
of the conjugates, we consider here the problem of simultaneous
approximation of a real number by all (resp. all but one) conjugates
of an algebraic number (resp. algebraic integer). Upon defining the
\emph{height} \Hp\ of a polynomial $P \in \R[T]$ to be the largest
absolute value of its coefficients, and the \emph{height} \Ha\ of an
algebraic number $\alpha \in \mathbb{C}$ to be the height of its
irreducible polynomial in \ZT, our main result reads as follows.

\begin{theoremA}
Let $\xi \in \R \backslash \mathbb{Q}$ and let $n \in \mathbb{N}^*$.
There exist positive constants $c_1, c_2$ depending only on $\xi$
and $n$ with the following properties.

\begin{itemize}
\item[(i)] There are infinitely many algebraic numbers $\alpha$ of degree
$n$ such that
   \begin{equation} \label{thm_bornemax1}
        \max_{\alphabar}\left| \xi - \alphabar \right| \leq c_1 \Ha^{-2/n}
   \end{equation}
   where the maximum is taken over all conjugates $\alphabar$ of $\alpha$.
\item[(ii)] There are infinitely many algebraic integers $\alpha$ of
degree $n+1$ such that
\begin{equation} \label{thm_bornemax2}
        \max_{\alphabar \neq \alpha}\left| \xi - \alphabar \right|
        \leq c_2 \Ha^{-2/n}
\end{equation}
   where the maximum is taken over all conjugates $\alphabar$
   different from $\alpha$.
\end{itemize}
\end{theoremA}

In the case $n=2$, this improves the estimates of the corollary in
Section 1 of \cite{AR}. In fact, as we will see in the next section,
the statement of part (i) is optimal up to the value of $c_1$ for
each $\xi \in \R \backslash \mathbb{Q}$, while the statement of part
(ii) is optimal up to the value of $c_2$ at least for quadratic
irrational values of $\xi$. This seems to be the first instance
where an optimal exponent of approximation is known for all values
of the degree $n$ in this type of problem.  The fact that we can
control the degree of the approximations originates from an
observation of Y.~Bugeaud and O.~Teuli\'e in \cite{BT}.

An irrational real number $\xi$ is said to be \emph{badly
approximable} if there exists a constant $c>0$ such that $|\xi -
p/q| \geq c q^{-2}$ for any rational number $p/q$. This is
equivalent to asking that $\xi$ has bounded partial quotients in its
continued fraction expansion (see Theorem 5F in Chapter 1 of
\cite{Sc}). For these numbers, we can refine Theorem A as follows.

\begin{theoremB}
Let $\xi \in \R \backslash \mathbb{Q}$ be badly approximable and let
$n \in \mathbb{N}^*$. Then there exist positive constants $c_1,
\ldots, c_4$ depending only on $\xi$ and $n$ with the following
properties. For each real number $X\geq1$, there is an algebraic
number $\alpha$ of degree $n$ satisfying (\ref{thm_bornemax1}) and
$c_3 X \leq \Ha \leq c_4 X$. There is also an algebraic integer
$\alpha$ of degree $n+1$ satisfying (\ref{thm_bornemax2}) and $c_3 X
\leq \Ha \leq c_4 X$.
\end{theoremB}

The proof of both results follows the method introduced by Davenport
and Schmidt in \cite{DS}. Let \RTn\ denote the real vector space of
polynomials of degree $\leq n$ in $\R[T]$, and let \ZTn\ denote the
subgroup of polynomials with integral coefficients in \RTn. We first
provide estimates for the last minimum of certain convex bodies of
\RTn\ with respect to \ZTn\ and then deduce the existence of
polynomials of \ZTn\ with specific inhomogeneous Diophantine
properties. This is done in Section 3. In Section 4, we show that
these polynomials have roots which fulfil the requirements of
Theorem A or B.

Throughout this paper, all implied constants in the Vinogradov
symbols $\gg$, $\ll$ and their conjunction $\asymp$ depend only on
$\xi$ and $n$.

\section{Optimality of the exponents of approximation}

Let $\xi \in \R \backslash \mathbb{Q}$ and let $n \in \mathbb{N}^*$.
If $n\ge 2$, the result in part (i) of Theorem A is optimal
up to the value of the implied constant since, for any algebraic
number $\alpha$ of degree $n$ with conjugates $\alpha_1, \ldots,
\alpha_n$, the discriminant $D(\alpha)$ of $\alpha$ satisfies
\[
\left|D(\alpha)\right|
  \leq \Ha^{2(n-1)} \prod_{1\leq i< j \leq n}
  \left| \alpha_i - \alpha_j \right|^2 \leq \Ha^{2(n-1)} \left( 2
\max_{1\leq i \leq n} |\xi -\alpha_i| \right)^{n(n-1)}
\]
Since $D(\alpha)$ is a non-zero integer, its absolute value is $\geq
1$, and thus we deduce that
\[
\max_{1\leq i \leq n} \left| \xi - \alpha_i \right| \geq
\frac{1}{2}\Ha^{-2/n}
\]
(compare with \S5 of \cite{Wi}).  If $n=1$, the result is
optimal for any badly approximable $\xi$.  Note that a similar
argument also shows that, for any algebraic integer $\alpha$ of
degree $n+1$ with conjugates $\alpha_1,\dots,\alpha_{n+1}$, we
have $\max_{1\leq i \leq n} \left| \xi - \alpha_i \right| \geq
(1/2)\Ha^{-2/(n-1)}$.

Similarly, the result in part (ii) of Theorem A is optimal up to the
value of the implied constant when $\xi$ is a quadratic irrational
number. To prove this, suppose that an algebraic integer $\alpha$ of
degree $n+1$ has conjugates $\alpha_1, \ldots, \alpha_{n+1}$
distinct from $\xi$ with the first $n$ satisfying
\[
\max_{1\leq i \leq n} \left| \xi - \alpha_i \right| \leq 1
\]
Let $Q(T) \in \ZT$ be the irreducible polynomial of $\xi$ over
$\mathbb{Z}$. Since $\alpha$ is an algebraic integer, the product
$Q(\alpha_1) \cdots Q(\alpha_{n+1})$ is a rational integer and since
it is non-zero (because $\xi$ is not a conjugate of $\alpha$), we
deduce that
\[
1 \leq \prod_{i=1}^{n+1}\left| Q(\alpha_i) \right|.
\]
For each $i=1, \ldots, n$, we have $|Q(\alpha_i)| \ll |\xi -
\alpha_i|$ since $\xi$ is a root of $Q$ and $|\xi - \alpha_i| \leq
1$. We also have $|Q(\alpha_{n+1})| \ll \max\{1, |\alpha_{n+1}|\}^2$
since $Q$ has degree 2. This gives
\[
1 \ll \Ha^2 \prod_{i=1}^{n} |\xi - \alpha_i|
\]
and consequently $\max_{1 \leq i \leq n}|\xi - \alpha_i| \gg
\Ha^{-2/n}$.


\begin{remark}
It would be interesting to know if there exists as well
transcendental numbers $\xi$ for which the exponent $2/n$ for \Ha\
in Theorem A part (ii) is best possible.
\end{remark}

\begin{remark}
The case where $\xi \in \mathbb{Q}$ is not interesting as it leads
to much weaker estimates. In this case, one finds that, for each
algebraic number $\alpha$ of degree n with $\alpha \neq \xi$, one
has $\max_{\alphabar}|\xi - \alphabar| \gg \Ha^{-1/n}$, and that,
for each algebraic integer $\alpha$ of degree $n+1$ with $\alpha
\neq \xi$, one has $\max_{\alphabar \neq \alpha}|\xi - \alphabar|
\gg \Ha^{-1/n}$.
\end{remark}

\section{Construction of polynomials}

Throughout this section, we fix an irrational real number $\xi \in
\mathbb{R} \backslash \mathbb{Q}$ and a positive integer $n \geq 1$.
For each integer $q \geq 1$, we denote by $\mathcal{C}(q)$ the
convex body of \RTn\ which consists of all polynomials $P\in \RTn$
satisfying
\[
|P^{[k]}(\xi)| \leq q^{2k-n} \ \ (0\leq k \leq n)
\]
where $P^{[k]}(\xi) = P^{(k)}(\xi)/k!$ denotes the $k$-th divided
derivative of $P$ at $\xi$ (the coefficient of $(T-\xi)^k$ in the
Taylor expansion of $P$ at $\xi$). We first prove :

\begin{proposition}
\label{prop_last_minimum} Let $q$ be the denominator of a convergent
of $\xi$. Then the last minimum of \Cq\ with respect to the lattice
\ZTn\ is $\leq 2^n$, and its first minimum is $\geq \left(
2^{n^2}(n+1)! \right)^{-1}$. Moreover, the convex body $2^n \Cq$
contains a basis of \ZTn\ over $\mathbb{Z}$.
\end{proposition}

\begin{proof}
Put $L_1 = qT - p$ where $p/q$ denotes a convergent of $\xi$ with
denominator $q$. If $q>1$, we also define $L_0 = q_0 T - p_0$ where
$p_0/q_0$ is the previous convergent of $\xi$ (in reduced form). If
$q=1$, we simply take $L_0 = 1$. The theory of continued fractions
tells us that these linear forms satisfy
\begin{equation}\label{eq_L1L2}
|L_i(\xi)| \leq q^{-1}\ \ \textrm{and} \ \ |L_i'(\xi)| \leq q
\end{equation}
for $i=0,1$, and moreover that their determinant (or Wronskian) is
$\pm 1$ (see \S 4 in Chapter I of [Sc]). The latter fact means that
$\{ L_0, L_1 \}$ spans $\mathbb{Z}[T]_{\leq 1}$ over $\mathbb{Z}$.
Therefore the products $P_j = L_0^j L_1^{n-j}\ (0 \leq j \leq n)$
span \ZTn\ over $\mathbb{Z}$ and, since the rank of \ZTn\ is $n+1$,
they form in fact a basis of \ZTn\ over $\mathbb{Z}$. Using
(\ref{eq_L1L2}), we also find that
\[
|P_j^{[k]}(\xi)| \leq \binom{n}{k} q^{2k-n} \leq 2^n q^{2k-n}\quad
(0\leq j,k \leq n).
\]
Thus $\{P_0, \ldots, P_n \}$ is a basis of \ZTn\ contained in $2^n
\Cq$. This proves the last assertion of the proposition as well as
the fact that the last minimum of \Cq\ is $\leq 2^n$.

Identify \RTn\ with $\mathbb{R}^{n+1}$ under the map which sends a
polynomial $a_0 + a_1 T + \ldots + a_n T^n$ to the point $(a_0, a_1,
\ldots, a_n)$. Then the linear map $\theta : \RTn \rightarrow
\mathbb{R}^{n+1}$ given by $\theta(P) = \left(P(\xi), P^{[1]}(\xi),
\ldots, P^{[n]}(\xi)\right)$ has determinant 1 and so \Cq\ has
volume $\prod_{k=0}^{n}\left( 2q^{2k-n} \right) = 2^{n+1}$. Since
the lattice $\ZTn$ has co-volume 1 (it is identified with
$\mathbb{Z}^{n+1}$), Minkowski`s second convex body theorem shows
that the successive minima $\lambda_1, \ldots, \lambda_{n+1}$ of
\Cq\ with respect to \ZTn\ satisfy $\left( (n+1)! \right)^{-1} \leq
\lambda_1 \cdots \lambda_{n+1} \leq 1$. Since $\lambda_2 \leq \ldots
\leq \lambda_{n+1} \leq 2^n$, this implies that $\lambda_1 \geq
\left( 2^{n^2}(n+1)! \right)^{-1}$.
\end{proof}

The construction of polynomials given by the next proposition uses
only the last assertion of Proposition \ref{prop_last_minimum}.

\begin{proposition}
\label{prop_exist_irred} Let $q$ be the denominator of a convergent
of $\xi$. There exist an irreducible polynomial $P(T) \in \ZT$ of
degree n and an irreducible monic polynomial $Q(T) \in \ZT$ of
degree $n+1$ satisfying
\[
c_5 q^{2k-n} \leq |P^{[k]}(\xi)|, |Q^{[k]}(\xi)| \leq 3 c_5 q^{2k-n}
\quad (0 \leq k \leq n)
\]
where $c_5 = (n+1) 2^{n+1}$.
\end{proposition}

Note that such polynomials have height $\asymp q^n$.

\begin{proof}
The last assertion of Proposition \ref{prop_last_minimum} tells us
the existence of a basis $\{P_0, \ldots, P_{n} \}$ of \ZTn\
satisfying
\begin{equation} \label{eq_borne_sup_p}
|P_j^{[k]}(\xi)| \leq 2^n q^{2k-n} \quad (0 \leq j, k \leq n).
\end{equation}
Since $\{P_0, \ldots, P_{n} \}$ is a basis of \ZTn\ over
$\mathbb{Z}$, we can write $T^n + 2 = \sum_{j=0}^n b_j P_j(T)$ for
some $b_0, \ldots, b_n \in \mathbb{Z}$. Consider the polynomial
\[
R(T) = 2 c_5 \sum_{k=0}^n q^{2k-n} (T-\xi)^k
\]
where $c_5 = (n+1) 2^{n+1}$. Since $\{P_0, \ldots, P_{n} \}$ is also
a basis of \RTn\ over $\mathbb{R}$, we can also write $R(T) =
\sum_{j=0}^n \theta_j P_j(T)$ for some $\theta_0, \ldots, \theta_n
\in \mathbb{R}$. Choose integers $a_0, \ldots, a_n$ such that $a_j
\equiv b_j \mod 4$ and $|a_j - \theta_j| \leq 2$ for $j=0, \ldots,
n$, and define $P(T) = \sum_{j=0}^n a_j P_j(T)$.

By construction $P(T)$ belongs to $\ZTn$ and is congruent to $T^n +
2$ modulo 4. Thus it is a polynomial of degree $n$ over $\mathbb{Q}$
 and it is irreducible by virtue of Eisenstein's criterion (for the prime 2).
Since $P(T) - R(T) = \sum_{j=0}^n \left( a_j - \theta_j \right)
P_j(T)$, we deduce from (\ref{eq_borne_sup_p}) that
\[
|P^{[k]}(\xi) - R^{[k]}(\xi)| \leq \sum_{j=0}^n \left| a_j -
\theta_j \right| \left| P_j^{[k]}(\xi) \right| \leq c_5 q^{2k-n}
\quad (0 \leq k \leq n).
\]
Since $R^{[k]}(\xi) = 2 c_5 q^{2k-n}$, it follows that $c_5 q^{2k-n}
\leq |P^{[k]}(\xi)| \leq 3 c_5 q^{2k-n}$ for $k=0,\ldots, n$, as
required.

The construction of $Q(T)$ is similar. Write
\[
T^{n+1} + 2 = T^{n+1} + \sum_{j=0}^n b_j' P_j(T) \quad \textrm{and}
\quad (T-\xi)^{n+1} + R(T) = T^{n+1} + \sum_{j=0}^n \theta_j'
P_j(T),
\]
with $b_0', \ldots, b_n' \in \mathbb{Z}$ and $\theta_0', \ldots,
\theta_n' \in \mathbb{R}$, and choose integers $a_0', \ldots, a_n'$
such that $a_j' \equiv b_j' \mod{4}$ and $|a_j' - \theta_j'| \leq 2$
for $j=0, \ldots, n$. Then the polynomial
\[
Q(T) = T^{n+1} + \sum_{j=0}^n a_j' P_j(T) \in \ZT
\]
is irreducible (by virtue of Eisenstein's criterion for 2), monic of
degree $n+1$, and it satisfies also $|Q^{[k]}(\xi) - R^{[k]}(\xi)|
\leq c_5 q^{2k-n}$ for $k=0, \ldots, n$.
\end{proof}
%
%
\section{Proof of Theorems A and B}
In this section, we prove the main theorems A and B of the
introduction by combining Proposition \ref{prop_exist_irred} with
the following result.

\begin{proposition}\label{prop_equivalence_racines}
Let $\xi \in \mathbb{R}$, let $n \in \mathbb{N}^*$, let $\delta > 0$
and let \cP\ be a subset of \ZT\ . Suppose that the elements of \cP\
are either polynomials of degree $n$ or monic polynomials of degree
$n+1$. Then the following conditions are equivalent :
\begin{itemize}
\item[(i)] There exists a constant $c_6 > 0$ such that
$|P^{[k]}(\xi)| \leq c_6 \Hp^{1-(n-k)\delta}$ for each $P \in \cP$
and each $k=0,1, \ldots, n$.
\item[(ii)] There exists a constant $c_7 > 0$ such that $|\xi - \alpha| \leq
c_7 \Hp^{-\delta}$ for each $P \in \cP$ and for $n$ of the roots
$\alpha$ of $P$, counting multiplicity.
\end{itemize}
\end{proposition}

\begin{proof}
Fix $P \in \cP$ and write it in the form
\[
P(T) = a_0(T-\alpha_1) \cdots (T-\alpha_m)
\]
where $m=\deg{P}$ and $\alpha_1, \ldots, \alpha_m$ are the roots of
$P$ ordered so that $|\xi-\alpha_1| \leq \ldots \leq
|\xi-\alpha_m|$. We put $\varepsilon = \Hp^{-\delta}$ and consider
the polynomial
\[
R(T) = P(\varepsilon T + \xi) = a_0 \varepsilon^m \prod_{k=1}^m
\left( T + \varepsilon^{-1}\left( \xi - \alpha_k \right) \right).
\]
The height of R is
\[
H(R) = \max_{0\leq k \leq m} \left| R^{[k]}(0)\right| = \max_{0 \leq
k \leq m} \left| P^{[k]}(\xi) \right| \varepsilon^k,
\]
and its Mahler measure is
\[
M(R) = |a_0| \varepsilon^m \prod_{k=1}^m \max\left\{1,
\varepsilon^{-1}|\xi - \alpha_k| \right\} = |a_0| \prod_{k=1}^m
\max\left\{\varepsilon, |\xi - \alpha_k| \right\}.
\]
For convenience, we also define
\[
L = \left\{ \begin{array}{ll} |a_0| & \textrm{if} \quad m=n \\
\max\left\{ \varepsilon, |\xi - \alpha_m| \right\} & \textrm{if}
\quad m=n+1 \end{array} \right .
\]
so that the formula for $M(R)$ becomes
\[
M(R) = L \prod_{k=1}^n \max\left\{ \varepsilon, |\xi - \alpha_k|
\right\}
\]
(recall that $a_0 = 1$ when $m=n+1$). Our argument below is based on
the standard inequalities relating these notions of heights, namely
\[
M(R) \leq (m+1) H(R) \quad \quad \textrm{and} \quad \quad H(R) \leq
2^m M(R).
\]
If condition (ii) holds, we find that $M(R) \leq c_7^n \varepsilon^n
L$. We also have $L \ll H(P)$ since $|a_0| \leq H(P)$ and since
$|\xi - \alpha| \ll \max\left\{ 1, |\alpha|\right\} \ll H(P)$ for
any root $\alpha$ of $P$. Then, for each $k=0, \ldots, n$, we obtain
\[
|P^{[k]}(\xi)| \ll \varepsilon^{-k} H(R) \ll \varepsilon^{-k} M(R)
\ll \varepsilon^{n-k} H(P)
\]
which shows that condition (i) holds.

Conversely assume that condition (i) holds. In this case we find
that $H(R) \leq c_6 \varepsilon^n H(P)$. We claim that $H(P) \ll L$.
If we take this for granted, we deduce that
\[
L \varepsilon^{n-1} |\xi - \alpha_n| \leq M(R) \ll H(R) \ll
\varepsilon^n L
\]
which implies that condition (ii) holds.

To prove the claim, we observe that
\[
H(P) \asymp H(P(T+\xi)) = \max_{0 \leq k \leq m} |P^{[k]}(\xi)|.
\]
By hypothesis, we have $|P^{[k]}(\xi)| \leq c_6 H(P)^{1-\delta}$ for
$k=0, \ldots, n-1$ and we also have $|P^{[m]}(\xi)| = 1$ if $m=n+1$.
Finally, we have $|P^{[n]}(\xi)| = |a_0|$ if $m=n$, and
$|P^{[n]}(\xi)| = |\sum_{k=1}^m (\xi - \alpha_k) | \leq m |\xi -
\alpha_m|$ if $m=n+1$, showing that $|P^{[n]}(\xi)| \ll L$. All this
implies that
\[
H(P) \ll \max\{1,L\}.
\]
Since $L \geq \varepsilon = \Hp^{-\delta}$, this in turn implies
that $H(P) \ll L$.
\end{proof}

\begin{proof}[Proof of the theorems]
Let $\xi \in \R \setminus \mathbb{Q}$ and $n \in \mathbb{N}^*$. We
simply prove Part (ii) of Theorems A and B since the proof of Part
(i) is similar and slightly easier.

For each denominator $q$ of a convergent of $\xi$, Proposition
\ref{prop_exist_irred} shows the existence of an irreducible monic
polynomial $Q \in \ZT$ of degree $n+1$ satisfying $H(Q) \asymp q^n$
and
\[
|Q^{[k]}(\xi)| \leq c_6 H(Q)^{(2k-n)/n} = c_6 H(Q)^{1-(n-k)(2/n)},
\quad (0 \leq k \leq n)
\]
for some constant $c_6 = c_6(\xi, n)$. The family $\mathcal{P}$ of
these polynomials satisfies the condition (i) of Proposition
$\ref{prop_equivalence_racines}$ for the choice $\delta = 2/n$, and
so it satisfies also the condition (ii) of the same proposition for
the same value of $\delta$ and for some constant $c_7$. For each $Q
\in \mathcal{P}$, choose a root $\alpha$ of $Q$ for which $|\xi -
\alpha|$ is maximal. Since $Q$ is irreducible, this root $\alpha$ is
an algebraic integer of degree $n+1$ and height $H(\alpha) = H(Q)$
whose conjugates $\overline{\alpha}$ over $\mathbb{Q}$ are the $n+1$
distinct roots of $Q$. Therefore, we get $\max_{\overline{\alpha}
\neq \alpha}|\xi - \overline{\alpha} | \leq c_7 H(\alpha)^{-2/n}$.
This proves Part (ii) of Theorem A since we find infinitely many
such numbers $\alpha$ by varying $Q$.

If $\xi$ is badly approximable, the ratios of the denominators of
consecutive convergents of $\xi$ are bounded. Thus, for each $X \geq
1$, there exists such a denominator $q$ with $q \asymp X^{1/n}$, and
so there exists a polynomial $Q \in \mathcal{P}$ with $H(Q) \asymp
X$. Consequently, the root $\alpha$ of $Q$ that we chose above
satisfies $H(\alpha) \asymp X$ and this proves Part (ii) of Theorem
B.
\end{proof}

\medskip
\noindent{\bf Acknowledgments.} The author thanks his MSc thesis
supervisor Damien Roy for suggesting this problem and for his help
in writing the present paper.

\end{document}